\documentclass[a4paper,centertags,osxzneside,12pt]{amsart}
\usepackage{mathrsfs}
\usepackage{appendix}
\usepackage{amssymb}
\usepackage{fancyhdr}
\usepackage{charter}
\usepackage{typearea}
\usepackage{pdfsync}
\usepackage{color}
\usepackage[colorlinks,linkcolor=blue,anchorcolor=blue,citecolor=green]{hyperref}
\usepackage{mathrsfs}
\usepackage[a4paper,top=3cm,bottom=3cm,left=3cm,right=3cm]{geometry}

\usepackage{enumitem}
\usepackage{color}
\setlist[1]{itemsep=5pt}
\newcommand{\comment}[1]{}

\makeatletter
      \def\@setcopyright{}
      \def\serieslogo@{}
      \makeatother

\newtheorem{theorem}{Theorem}[section]

\newtheorem{remark}[theorem]{Remark}

\numberwithin{equation}{section}
\begin{document}
\title{A note on csc Bergman metric}
\author{Xiaojun Huang}
\address{Department of Mathematics, Rutgers University, New Brunswick, NJ 08903, USA.}
\email{huangx@math.rutgers.edu}
\thanks{Xiaojun Huang is partially supported  by NSF DMS-2247151}

\author{Xiaoshan Li}
\address{School of Mathematics and Statistics, Wuhan University, Wuhan, Hubei 430072, China.}
\email{xiaoshanli@whu.edu.cn}
\thanks{Xiaoshan Li is supported in part by NSFC (12361131577, 12271411)}
\date{April 2026}
\begin{abstract}
    In this note, we show that if the Bergman metric of a pseudoconvex domain in $\mathbb C^n$($n\geq 3$) has constant scalar curvature, then every strongly pseudoconvex boundary point of the domain is spherical.
\end{abstract}
\maketitle
\section{csc Bergman metric}

Let $\Omega\subset\mathbb C^n$ ($n\geq 1$) be a possibly unbounded pseudoconvex domain containing a smooth strongly pseudoconvex boundary point $p\in\partial\Omega$. Write $\{\phi_j\}$ for an orthonormal basis of the Bergman space $A^2(\Omega)$, the subspace of $L^2(\Omega)$ consisting of $L^2$-integrable holomorphic functions on $\Omega$. Denote by
$
K_{\Omega}(z, z)=\sum_{j}|\phi_j(z)|^2
$
the Bergman kernel function on $\Omega$, and let
\[
g_{\Omega}=\sum_{i,j}g_{i\overline j}\,dz_i\otimes d\overline z_j, \ \ g_{i\overline j}=\frac{\partial^2}{\partial z_i\,\partial \overline z_j}\log K_{\Omega}(z,z)
\]
be the  Bergman metric of $\Omega$. The Bergman metric is well defined on a maximal open subset $\Omega^*\subset \Omega$ that contains a one‑sided neighborhood of $p$ (see~\cite{HJL25}). 

Write $G_{\Omega}:=(g_{i\overline j})_{n\times n}$ and denote
$
J_{\Omega}=\frac{\det G_{\Omega}}{K_{\Omega}},
$ 
called the Bergman canonical invariant function of $\Omega$. 
For the Bergman metric with metric tensor \( g_{i\bar{j}} \), the Ricci tensor is given by
$ 
\mathrm{Ric}_{i\bar{j}}
=
-\,\frac{\partial^2}{\partial z^i \,\partial \bar{z}^j}
\log \det\bigl(g_{k\bar{\ell}}\bigr).
$ 
The scalar curvature is the trace of the Ricci tensor with respect to the metric:
$ 
S_{\Omega}=g^{i\bar{j}} \,\mathrm{Ric}_{i\bar{j}}=
-\,g^{i\bar{j}} \,\frac{\partial^2}{\partial z^i \,\partial \bar{z}^j}
\log \det\bigl(g_{k\bar{\ell}}\bigr).
$ 
It is known that near a strongly pseudoconvex boundary point \cite{KYu96}, the scalar curvature \( S_{\Omega} \) and the invariant \( J_\Omega \) approach the limits \( -n \) and the constant \( c_n=\frac{(n+1)^n\pi^n}{n!} \), respectively. Moreover, the Bergman metric is asymptotically K\"ahler--Einstein with Ricci constant \( -1 \) in the sense that $\mathrm{Ric}_{i\bar{j}}+g_{i\bar{j}}\rightarrow 0$ as $z\rightarrow p$.

Starting from the identity
$ 
\log J_\Omega = \log \det G_\Omega - \log K_\Omega,
$ 
and applying \( \partial_{z_k} \partial_{\overline{z_j}} \), then contracting with \( g^{k\bar{j}} \), we obtain that \( S_{\Omega} \) is constant if and only if \( \log J_\Omega \) is harmonic with respect to the Bergman metric, which was first derived by Sha in  (\cite{S26}):
\[
\Delta_{g_\Omega} \log J_\Omega (z)
=
\sum_{j,k} g^{k\bar{j}}
\frac{\partial^2}{\partial z^k \,\partial \overline{z^j}} \log J_\Omega=0
\quad \text{on } \Omega^*.
\]

When the Bergman metric \( g_\Omega \) has constant scalar curvature on \( \Omega^* \), we say that \( \Omega \) admits a constant scalar curvature ({\bf csc}) Bergman metric.
Since this condition can be expressed as a real-analytic equation on $K(z,z)$ in \( \Omega \), it follows from the uniqueness of real analytic functions  that \( \Omega \) has a csc Bergman metric if and only if \( S_\Omega \) is constant on some nonempty open subset of \( \Omega^* \).

In this note, we present a   proof of the following fact:
\begin{theorem} \label{main}
Let $\Omega
\subset {\mathbb C}^n$ with $n\ge 3$ be a pseudoconvex domain with $p\in \partial\Omega$ a  smooth strongly pseudoconvex  boundary point of $\Omega$. If $\Omega$  has  a {\bf csc} Bergman metric then $\partial\Omega$ is locally spherical near $p$, namely, a small open piece of 
$\partial\Omega$ is CR-diffeomorphic to an open piece of the boundary of the unit ball $\mathbb{B}^n$. 
\end{theorem}

\section{Proof of Theorem \ref{main}}

The proof is similar to that of the corresponding known result (see, e.g., \cite{HX21,HL23}), where the csc condition is replaced by the slightly stronger K\"ahler--Einstein condition. A new observation here is the use of a recently obtained formula of Martin for the expansion of $J_\Omega$ \cite{Ma21}, in addition to Christoffers's formula for the Bergman expansion \cite{Ch81}. We also use several asymptotic expansion formulas of Engliš \cite{Eng08} to simplify the computation.
\begin{proof}
Let $G\subset\Omega$ be a small smoothly bounded  strongly pseudoconvex domain with $U\cap G=U\cap \Omega$ for some neighborhood $U$ of the strongly pseudoconvex boundary point $p$ in $\mathbb C^n$. By localization of Bergman kernels (see \cite{HHL26}), we have
\begin{equation}\label{localization}
    K_\Omega(z, z)=K_G(z, z)+\varphi(z),
\end{equation}
where $\varphi(z)\in C^\infty(U\cap\overline G)$. Let $\rho\in C^\infty(\overline G)$ be a positively signed Fefferman defining function for $G$, namely, $J_{MA}(\rho)=1+O(\rho^{n+1})$ with the Monge-Ampere operator $J_{MA}(\rho)$ defined in the following (\ref{MA}). With respect to such a $\rho$, then we have  the Fefferman expansion of $K_G(z, z)$ as follows:
\begin{equation}\label{expansion}
    K_G(z, z)=\frac{\phi}{\rho^{n+1}}+\psi\log\rho
\end{equation}
with $\phi, \psi\in C^\infty(\overline G)$ and $\phi=\frac{n!}{\pi^n}+O(\rho^2)$. It follows from (\ref{localization}) and (\ref{expansion}) that 
\begin{equation}\label{AP}
    K_{\Omega}(z, z)=\frac{\phi+\psi\rho^{n+1}\log\rho+\varphi\rho^{n+1}}{\rho^{n+1}}=\frac{\Phi}{\rho^{n+1}},
\end{equation}
where $\Phi=\phi+\psi\rho^{n+1}\log\rho+\varphi\rho^{n+1}\in C^n(U\cap\overline G)$. Moreover, $\Phi=\phi+o(\rho^n)$ on $U\cap\overline G$.
Let $a, b$ be smooth functions on $\overline G$ such that 
\begin{equation*}
    \phi=\frac{n!}{\pi^n}(1+a\rho^2+b\rho^3+O(\rho^4)).
\end{equation*}
Here, $a$ is uniquely determined up to $O(\rho^3)$ near $p$. Then by a result of Christoffers \cite{Ch81}, $a(p)=0$ if and only if $p$ is a CR umbilic point of $\partial G$ in the sense of Chern-Moser \cite{CM74}. Write $u_{\Omega}=(\frac{\pi^n}{n!}K_{\Omega}(z, z))^{\frac{-1}{n+1}}, z\in\Omega^\ast$ and $u_G=(\frac{\pi^n}{n!}K_{G}(z, z))^{\frac{-1}{n+1}}, z\in G$. For any real-valued  $C^2$-smooth function $u$, define
\begin{equation}\label{MA}
J_{MA}(u)=(-1)^n\det\left[\begin{array}{cc}
     u & u_{\overline\beta}  \\
     u_{\alpha}&  u_{\alpha\overline\beta}
\end{array}\right]
\end{equation}
called  the Fefferman-Monge-Ampere operator. By a result of Martin \cite[Page 93]{Ma21}, 
we have
$$J_{MA}(u_{G})=1-3\frac{n-1}{n+1}a\rho^2+o(\rho^2)~\text{on} ~U\cap \overline G.$$
On the other hand, applying the following Fefferman formula \cite{Ma21} 
\begin{equation}\label{Fefferman formula}
    J_{G}=\frac{(n+1)^n\pi^n}{n!} J_{MA}(u_{G}):=c_n J_{MA}(u_{G}) \ \ \hbox{where  } c_n=\frac{(n+1)^n\pi^n}{n!}
\end{equation}
 we have
\begin{equation}
    J_G(z)=c_n-3c_n\frac{n-1}{n+1}a\rho^2+o(\rho^2).
\end{equation}
 
Again by the Fefferman formula, we verify that 
$J_\Omega=J_G+O(\rho^3)$ near $p$. Thus 
\begin{equation}
    J_{\Omega}(z)=c_n-3c_n\frac{n-1}{n+1}a\rho^2+o(\rho^2)~\text{near }p.
\end{equation}

Since $J_\Omega(z)$  is independent of the chosen defining function and two defining functions differ by a positive smooth function in a neighborhood of $p$, we may assume, without loss of generality,  in the following computation that $\rho$ is strongly plurisubharmonic near $p$ to prove $a(p) = 0$.
By (\ref{AP}), we have the following expansion of the Bergman canonical invariant function from \cite{Eng08} (Engliš stated this result for strongly pseudoconvex domains, with (\ref{AP}),  however, his proof carries over without any change to our local setting):

\begin{equation}
\begin{split}
    J_{\Omega}(z)&\sim\sum_{j=0}^\infty(\rho^{n+1}\log\rho)^j\eta_j, \eta_j\in C^\infty(U\cap\overline G)\\
    &    \sim\eta_0+\eta_1\rho^{n+1}\log\rho+\cdots.
    \end{split}
\end{equation}
Here, the sum is in the asymptotic sense, that is, for any $k\in\mathbb N$, the difference $$J_\Omega(z)-\sum_{j=0}^{k-1}(\rho^{n+1}\log\rho)^j\eta_j\in C^{k(n+1)-1}(U\cap\overline G)$$ and vanishes on $U\cap\partial G$ with all its partial derivatives of orders $\leq k(n+1)-1$. 
Then we can write 
\begin{equation}
   \log( J_{\Omega}(z))=\log(c_n)-3\frac{n-1}{n+1}a(z)\rho^2+b(z)\rho^2
\end{equation}
where $b(z)\in C^\infty(U\cap G)\cap C^{1,\frac{1}{2}}(U\cap\overline G)$ and 
\begin{equation}\label{5-24-a1}
b(z)\sim a_1(z)\rho+\rho^{-2}\sum_{j=1}^\infty \tilde \eta_{j}(\rho^{n+1}\log\rho)^j,
\end{equation}
with $\tilde \eta_k(z)\in C^\infty(U\cap\overline G)$ for $k\ge 1.$

Since 
 $\Delta_{g_\Omega}\log J_{\Omega}=0$,  we have 
\begin{equation}\label{5-23-a3}
\sum_{i, j=1}^ng^{\overline ji}\frac{\partial^2}{\partial z_i\partial\overline z_j}[3a\frac{n-1}{n+1}\rho^2- b\rho^2]\equiv 0~\text{on} ~U\cap G.
\end{equation}

By direct calculation,
\begin{equation}\label{5-23-a1}
\begin{split}
     g^{\overline ji}\frac{\partial^2}{\partial z_i\partial\overline z_j}(a\rho^2)=&\rho^2 g^{\overline j i}\frac{\partial^2 a}{\partial z_i\partial\overline z_j}+2\rho g^{\overline j i}\left(\frac{\partial a}{\partial z_i}\frac{\partial\rho}{\partial \overline z_j}+\frac{\partial a}{\partial \overline z_j}\frac{\partial \rho}{\partial z_i}\right)+\\
    &2a g^{\overline j i}\frac{\partial \rho}{\partial\overline z_j}\frac{\partial\rho}{\partial z_i}+2a\rho g^{\overline ji}\frac{\partial^2 \rho}{\partial z_i\partial\overline z_j}.
    \end{split}
\end{equation}
Since we can write $K_{\Omega}=\rho^{-(n+1)}[\widetilde \eta_0+(\rho^{n+1}\log\rho)\widetilde\eta_1],$
where $\widetilde\eta_0=\phi+\varphi\rho^{n+1}$ and $\widetilde\eta_1=\psi$, from  \cite[Theorem 3]{Eng08}, we have 
\begin{equation}
    \rho^{-1} g^{\overline j i}\in C^n(U\cap\overline G).
\end{equation}
Moreover, from \cite[Section 3]{Eng08} we have in  a small neighborhood $U_p\subset {\mathbb C}^n$ of $p$
\begin{equation}\label{5-23-a2}
\begin{split}
    &\rho^{-1}g^{\overline j i}=\rho^{-1}[\log \rho]^{\overline j l}H^i_l, H^i_l\in C^n(U_p\cap\overline G), H^i_l|_{\partial G\cap U_p}=-\frac{1}{n+1}\delta^i_l.\\
    &\frac{1}{\rho^2}[\log\rho]^{\overline j i}\rho_i, ~~\frac{1}{\rho^2}[\log\rho]^{\overline j i}\rho_{\overline j}\in C^\infty(\overline G\cap U_p),~[\log \rho]^{\overline j i}\rho_{\overline j}\rho_i=\rho [\log \rho]^{\overline j i}\rho_{i\overline j}-n\rho^2~\text{on}~U_p\cap\overline G.\\
    &\lim_{z\rightarrow p} \frac{1}{\rho^2}[\log \rho]^{\overline j i}\rho_{\overline j}\rho_i=-1.
    \end{split}
\end{equation}
Replacing $a(z)$ by $b(z)$ from (\ref{5-24-a1}) in (\ref{5-23-a1}) and from (\ref{5-23-a2}) we have that
\begin{equation}
    \lim_{z\rightarrow p}\frac{1}{\rho^2}\sum_{i, j=1}^n g^{\overline ji}\frac{\partial^2}{\partial z_i\partial\overline z_j}(b\rho^2)=0.
\end{equation}
It follows from (\ref{5-23-a3}), (\ref{5-23-a1}) and (\ref{5-23-a2}) that 
\begin{equation}
    0=\lim_{z\rightarrow p}\frac{1}{\rho^2}\sum_{i, j=1}^n g^{\overline ji}\frac{\partial^2}{\partial z_i\partial\overline z_j}(a\rho^2)=2a(p)(n-2).
\end{equation}
Since $n\geq 3$, we have $a(p)=0$ and thus an arbitrary given strongly pseueoconvex point $p$ is a CR umbilic point of $\partial\Omega$.
It follows immediately from the Chern-Moser Lemma \cite{CM74} that $\partial\Omega$ is spherical near $p$.
\end{proof}

\section{Remarks}

%We make the following more remarks in this section.
\begin{remark}\label{rem1}
    For $m>1$ an integer  and $n\geq 2$, let $$\mathcal E_m:=\left\{z\in\mathbb C^{n-1}\times \mathbb C:\sum_{j=1}^{n-1}|z_j|^2+|z_n|^{2m}<1\right\}.$$ Then we claim that the Bergman metric of $\mathcal E_m$ can not have constant scaler curvature.
\end{remark}
\begin{proof}[Proof of the statement in Remark \ref{rem1}]
    We denote by  $S$  the scalar curvature of the Bergman metric of $\mathcal E_m$. If $\mathcal E_m$ has a csc Bergman metric, then $S\equiv -n$ on $\mathcal E_m$ by \cite[Corollary 2]{KYu96} as it has limit $-n$ at any strongly pseudoconvex point. From \cite {KYu96}, 
    \begin{equation}\label{scaler}
    S(0)=n(n+1)-4a_0\sum_{j=1}^n\frac{b_{jj}}{a_j^2}-a_0\sum_{j\neq k}\frac{b_{jk}}{a_ja_k}.
    \end{equation}
    Here, $a_0=\frac{1}{vol(\mathcal E_m)}$, $a_j=\frac{1}{\|z_j\|_{\mathcal E_m}^2}$, $b_{jk}=\frac{1}{\|z_jz_k\|^2_{\mathcal E_m}}$ with $\|\cdot\|$ the $L^2$-norm on $\mathcal E_m$.
    By direct calculations, 
    \begin{equation}\label{5-28-a1}
\begin{aligned}
a_0=\frac{m\Gamma\!\left(n + \frac{1}{m}\right)}{\pi^{n} \; \Gamma\!\left(\frac{1}{m}\right)},~
a_1=\dots=a_{n-1}=\frac{m\Gamma\!\left(n +1+ \frac{1}{m}\right)}{\pi^{n} \; \Gamma\!\left(\frac{1}{m}\right)},~ 
a_n=\frac{m\;\Gamma\!\left(n+\frac{2}{m}\right)}{\pi^{n} \; \Gamma\!\left(\frac{2}{m}\right)}
\end{aligned}
\end{equation}
and 
\begin{equation}\label{5-28-a2}
\begin{split}
&b_{jj}=\dfrac{m}{2\pi^n}\,\dfrac{\Gamma\!\left(n+2+\frac1m\right)}{\Gamma\!\left(\frac1m\right)},~ j=1,\dots,n-1; ~b_{nn}=\dfrac{m}{\pi^n}\,\dfrac{\Gamma\!\left(n+\frac3m\right)}{\Gamma\!\left(\frac3m\right)},\\
&b_{jk}=\dfrac{m}{\pi^n}\,\dfrac{\Gamma\!\left(n+2+\frac1m\right)}{\Gamma\!\left(\frac1m\right)},~  j\neq k,\; j,k=1,\dots,n-1,\\
&b_{jn}=b_{nj}=\dfrac{m}{\pi^n}\,\dfrac{\Gamma\!\left(n+1+\frac2m\right)}{\Gamma\!\left(\frac2m\right)},  j=1,\dots,n-1
\end{split}
 \end{equation}   
Substituting (\ref{5-28-a1}) and (\ref{5-28-a2}) to (\ref{scaler}) we have
\begin{equation}
S(0)=2-\frac{(n-1)(n+\frac 2m)}{n+\frac 1m}
-4\,\frac{\Gamma\!\left(n+\frac1m\right)\,\
\Gamma\!\left(n+\frac3m\right)\Gamma\!\left(\frac2m\right)^2\,}{\Gamma\!\left(\frac1m\right)\,\Gamma\!\left(\frac3m\right)\,
\Gamma\!\left(n+\frac2m\right)^2}.
\end{equation}
Set $a=\frac{1}{m}\in(0, 1)$, $H_n(a)
=
\frac{
\Gamma(n+3a)\Gamma(n+a)\Gamma(2a)^2
}{
\Gamma(3a)\Gamma(a)\Gamma(n+2a)^2
}=\prod_{j=0}^{n-1}
\frac{(j+a)(j+3a)}{(j+2a)^2}$ and $T_n(a)
=\frac{3n+(4-n)a}{4(n+a)}.$ Then 
\begin{equation}
\begin{split}
    S(0)+n=4(T_n(a)-H_n(a)).
    \end{split}
\end{equation}
By direct calculation $T_2(a)<H_2(a)$.  Since
$H_{n+1}(a)=H_n(a)
\frac{(n+a)(n+3a)}{(n+2a)^2}$ and 
$$
T_n(a)\frac{(n+a)(n+3a)}{(n+2a)^2}-T_{n+1}(a)=\frac{an(1-a)^2}{4(n+2a)^2(n+1+a)}.
$$
Since $0<a<1$ and $n\geq 2$, the right-hand side is strictly positive,
hence
\[T_n(a)\frac{(n+a)(n+3a)}{(n+2a)^2}>T_{n+1}(a).
\]
Assume inductively that
\[H_n(a)>T_n(a).\]
Then
\[H_{n+1}(a)=H_n(a)\frac{(n+a)(n+3a)}{(n+2a)^2}>T_n(a)\frac{(n+a)(n+3a)}{(n+2a)^2}>T_{n+1}(a).\]
Thus
\[H_n(a)>T_n(a)\]
for every $n\ge2$.
Consequently,
\[S(0)+n=4(T_n(a)-H_n(a))<0.\]
Therefore, $S(0)<-n$ and thus we get a contradiction.
\end{proof}

\begin{remark}
We remark that Theorem~\ref{main}, together with Remark~\ref{rem1},
allows us to restate some of the results in \cite{HHL26} 
% Theorem~1.1 and Corollary~1.4 of \cite{HHL26}
in a slightly different form. For instance,
Theorem~1.1 and Corollary~1.4 of \cite{HHL26} can now be stated as:  A bounded real analytic pseudoconvex domain  
or
a smoothly bounded  convex domain of finite D'Angelo type   
admits a {\bf csc} Bergman metric if and only if it is biholomorphic to the ball.
Indeed, the K\"ahler--Einstein condition in \cite{HHL26} is used
only to ensure sphericity at strongly pseudoconvex points; 
and the proof of Theorem~3.3 in \cite{HHL26} carries over verbatim to the
{\bf csc} Bergman metric setting by using the formula~$N_\Omega$ in Proposition~2.1~(iv) of \cite{KYu96} in place of~$\lambda_\Omega$, and by observing that the same localization 
results in Proposition~2.4 of \cite{KYu96} hold  for unbounded pseudoconvex domains as in \cite{HJL25}.
\end{remark}

{\bf Acknowledgement}. In a personal communication, M. Xiao informed us that he and his coauthors 
have a work in progress in which, among other things, they also obtain results closely related to Theorem 1.1. The  authors thank  Song-Ying Li for several  helpful discussions.

\end{document}